\documentclass[a4paper,11pt]{amsart}
\usepackage[T1]{fontenc} 
\usepackage{amsmath,amsthm}
\usepackage[francais,english]{babel}
\usepackage[dvips]{graphicx}
\usepackage{todonotes}
\usepackage{mathdots}
\usepackage{color}
\usepackage{tikz}
\usetikzlibrary{arrows,positioning,shapes,matrix,backgrounds}
\usetikzlibrary{fit,calc,decorations.pathreplacing}

\usepgflibrary{arrows}
\usepackage{soul}
\usepackage{amsfonts,amssymb}
\usepackage{geometry}
\usepackage{hyperref}
\usepackage{stmaryrd}
\usepackage{fancyhdr}
\usepackage{url}
\usepackage{todonotes}
\usepackage{dsfont}
\usepackage[utf8]{inputenc} 
\usepackage{enumerate}

\theoremstyle{definition}

\newtheorem*{thm*}{Theorem}

\newtheorem{prop}{Proposition}[section]

\newtheorem{lemma}[prop]{Lemma}

\newtheorem{thm}[prop]{Theorem}
\newtheorem{definition}[prop]{Definition}

\newtheorem{corollary}[prop]{Corollary}

\newtheoremstyle{pourlesremarques}{\topsep}{\topsep}{\normalfont}{}{\bfseries}{.}{ }{}
\theoremstyle{pourlesremarques}
\newtheorem{rem}[prop]{Remark}
\newtheorem*{rem*}{Remark}
\newtheoremstyle{pourlesexemples}{\topsep}{\topsep}{\normalfont}{}{\bfseries}{.}{ }{}
\theoremstyle{pourlesexemples}

\def\presuper#1#2%
  {\mathop{}%
   \mathopen{\vphantom{#2}}^{#1}%
   \kern-\scriptspace%
   #2}
   
\newcommand{\Wh}{\mathcal{W}}

\def\Rep{\operatorname{Rep}}

\newcommand{\Ind}{\operatorname{Ind}}
\newcommand{\ind}{\operatorname{ind}}

\newcommand{\Bil}{\operatorname{Bil}}

\newcommand{\Hom}{\operatorname{Hom}}
\renewcommand{\subset}{\subseteq}

\newcommand{\diag}{\operatorname{diag}}

\newcommand{\GL}{\operatorname{GL}}

\newcommand{\Sp}{\operatorname{Spec}}

\newcommand{\Res}{\operatorname{Res}}

\newcommand{\w}{\varpi}
\renewcommand{\d}{\delta}

\newcommand{\C}{\mathbb{C}}

\newcommand{\Q}{\mathbb{Q}}

\newcommand{\M}{\mathcal{M}}

\newcommand{\Z}{\mathbb{Z}}

\newcommand{\1}{\mathbf{1}}

\def\End{\operatorname{End}}

\def\val{\operatorname{val}}

\def\GL{\operatorname{GL}}

\def\\Hom{\operatorname{\Hom}}

\def\dim{\operatorname{dim}}

\def\Ql{\overline{\mathbb{Q}_{\ell}}}

\def\Rep{\operatorname{Rep}}
\def\Zl{\overline{\mathbb{Z}_{\ell}}}
\def\Fl{\overline{\mathbb{F}_{\ell}}}

\def\diag{\operatorname{diag}}
\def\leq{\leqslant}
\def\geq{\geqslant}

\def\Res{\operatorname{Res}}
\def\Rep{\operatorname{Mod}}

\def\presuper#1#2%
  {\mathop{}%
   \mathopen{\vphantom{#2}}^{#1}%
   \kern-\scriptspace%
   #2}
\makeatletter
\DeclareRobustCommand{\rvdots}{%
  \vbox{
    \baselineskip4\p@\lineskiplimit\z@
    \kern-\p@
    \hbox{.}\hbox{.}\hbox{.}  \hbox{.}
  }}
\makeatother

\subjclass{11F70 (primary); 11F85, 11S40 (secondary)}
\keywords{Whittaker model, Kirillov model, gamma factor, families, integral structure}

\title{The Kirillov model in families}
\author{N. Matringe}
\thanks{Nadir Matringe, Universit\'e de Poitiers, Laboratoire de Math\'ematiques et Applications, T\'el\'eport 2 - BP 30179, Boulevard Marie et Pierre Curie, 86962, Futuroscope Chasseneuil Cedex. France. \\
Email:~Nadir.Matringe@math.univ-poitiers.fr}
\author{G. Moss}
\thanks{Gilbert Moss, University of Utah, Department of Mathematics, 155 South 1400 East JWB233, Salt Lake City, Utah 84103. USA. \\ Email:~moss@math.utah.edu}

\date{\today}
\begin{document}
\maketitle

\begin{abstract}
\noindent Let $F$ be a non archimedean local field, let $k$ be an algebraically closed field of characteristic $\ell$ different from the residual characteristic 
of $F$, and let $A$ be a commutative Noetherian $W(k)$-algebra, where $W(k)$ denotes the Witt vectors. By extending recent results of the second author regarding the Rankin-Selberg functional equations, we show that if $V$ is an $A[\GL_n(F)]$-module of Whittaker type, then the mirabolic restriction map on its Whittaker {space} is injective. This gives a new quick proof of the existence of Kirillov models for representations of Whittaker type, including complex representations, which generalizes to the $\ell$-modular and families setting, in contrast with the previous proofs. In the special case where $A=k=\Fl$ and $V$ is irreducible generic, our result in particular answers a question of Vignéras from \cite{VigGL2}.
\end{abstract}

\section{Introduction}

Let $F$ be a nonarchimedean local field with residue characteristic $p$. Kirillov, in \cite{kirillov} for irreducible unitary complex representations acting on Hilbert spaces, and then Jacquet and Langlands in \cite{JL} for smooth irreducible complex representations, produced models for representations of $GL_2(F)$ in the space $\Ind_{N_2}^{P_2}\psi$ of locally constant functions $f$ on the mirabolic subgroup $P_2:=\left\{\left(\begin{smallmatrix}g & x\\ 0 & 1 \end{smallmatrix} \right)\right\}$ satisfying $f(nh)=\psi(n)f(h)$, $n\in N_2$, $h\in P_2$, where $N_2:=\left\{\left(\begin{smallmatrix}1 & x\\ 0 & 1 \end{smallmatrix} \right)\right\}$, and $\psi$ is a nondegenerate character of $N_2$. Gelfand and Kazhdan \cite{GK} showed that all cuspidal irreducible representations of $G_n:=GL_n(F)$ admit a Kirillov model, and conjectured that all \emph{generic} irreducibles do, where generic means having non-vanishing twisted coinvariants $\pi_{N_n,\psi}$ (when $n=2$ an irreducible representation is generic if and only if it is infinite-dimensional, but not in general). Bernstein and Zelevinsky proved this conjecture in \cite{BZ77} using their theory of ``derivatives,'' which express the representation theory of $P_n$ in terms of $G_m$ for $m< n$. Another proof was then given by Bernstein using methods of harmonic analysis in \cite{Ber}, and then Jacquet and Shalika gave a proof which applies to a more general class of representations which are nowadays called standard modules, in \cite{JSpacific}. The Kirillov model property has proved useful to study local constants, for instance it is used to establish properties of local new vectors and conductors of generic representations of $G_n$ in \cite{JSconductor,conductor_corrected,matringe_essential} (see \cite{CasselmanConductor} for $GL_2(F)$).

For the group $G_3$, a proof was given by Jacquet, Piatetski-Shapiro, and Shalika in \cite[Prop 7.5.1]{JPS1}, using their functional equation of the Rankin-Selberg local gamma factor. This proof was never generalized though it was observed in \cite{JPS1} that the same proof would work for any 
$n\geq 3$ given then $(n,n-2)$ functional equation, not available at that time. Our goal is to expand and in fact simplify the method of \cite[Prop 7.5.1]{JPS1}, using the $(n,n-1)$ functional equation, in order to establish the existence of Kirillov models for ``essentially''  generic representations and, more generally, for generic $\ell$-adic families. Essentially generic representations form a broader class than irreducible generic ones: we define them, following  \cite{EH14}, to be the finite-length representations with a single irreducible generic subquotient, which is positioned as a submodule. Up to isomorphism, they are the Whittaker spaces $\Wh(\pi,\psi)$ associated to representations $\pi$ of $\emph{Whittaker type}$ in the terminology of Jacquet, Piatetski-Shapiro, and Shalika (\cite{JPS2}), i.e., finite-length $\pi$ satisfying $\dim_{\mathbb{C}}(\pi_{U_n,\psi})=1$. {The existence of Kirillov models for essentially generic representations over $\mathbb{C}$ (which generalizes \cite{JSpacific}) first appeared in \cite[4.3]{EH14}, and can also be found in \cite{KMextension}. Using the idea of \cite[Prop 7.5.1]{JPS1}, we give a new proof, which is quicker than all previous proofs if the Rankin--Selberg functional equation is taken for granted. Moreover, our method works in the broader setting of $\ell$-adic families, where the coefficient field $\mathbb{C}$ is replaced by an $\ell$-adic ring, for a prime $\ell\neq p$.}

Some of the deepest number-theoretic properties of automorphic forms are expressed in terms of their congruences modulo various prime numbers $\ell$. In the local setting, this motivates the study of congruences among admissible representations of $GL_n(F)$, or more broadly, the deformation of representations in $\ell$-adic integral families. When $\ell\neq p$, it was observed in \cite{EH14}, guided by the local Langlands correspondence, that irreducible generic objects are not well-suited to deformation in families, and should be replaced by essentially generic ones. The existence of Kirillov models for essentially generic objects means the injectivity of the restriction-to-$P_n$ map,
\begin{equation}\label{kirillovinjectivity}
W\mapsto W|_{P_n}: \Ind_{N_n}^{G_n}\psi\to \Ind_{N_n}^{P_n}\psi,
\end{equation}
on the Whittaker spaces $\Wh(\pi,\psi)$ of representations $\pi$ of Whittaker type.

Let $k$ be an algebraically closed field of characteristic $\ell\neq p$, let $W(k)$ denote the ring of Witt vectors, and let $A$ be a commutative $W(k)$-algebra. A smooth $A[G_n]$-module is a ``family'' of representations over the base scheme $\Sp(A)$ in the sense that its fiber at each point is a smooth representation over the residue field at that point. Let $\psi:N_n\to W(k)^{\times}$ be a nondegenerate character and let $\psi_A: N_n\to A^{\times}$ be the character induced by $\psi$. Our first main result is that this injectivity of (\ref{kirillovinjectivity}) holds after replacing $\mathbb{C}$ with a $W(k)$-algebra $A$, and $\psi$ with $\psi_A$.
\begin{thm}\label{kirillovintro}
Let $A$ be a Noetherian $W(k)$-algebra, and let $\Wh$ be an $A[G_n]$-submodule of $\Ind_{N_n}^{G_n}\psi_A$ that is admissible, finitely generated over $A[G_n]$, such that $\Wh_{N_n,\psi}$ is free of rank one over $A$. Then the map $W\mapsto W|_{P_n}$ is injective on $\Wh$.
\end{thm}

If $\mathcal{K}=W(k)[\frac{1}{\ell}]$, and we choose an isomorphism $\overline{\mathcal{K}}\cong \mathbb{C}$, then taking the $W(k)$-algebra $A=\mathbb{C}$ gives the one-point family of a usual smooth representation over $\mathbb{C}$, and Theorem~\ref{kirillovintro} specializes to the existence of Kirillov models over $\mathbb{C}$ for essentially generic representations. 

In \cite{VigGL2}, Vign\'{e}ras used Kirillov models to define integral structures for irreducible generic representations of $GL_2(F)$, and asked whether it could be done for $G_n$. Vign\'{e}ras established existence of Kirillov models for cuspidal irreducible $k[G_n]$-modules in \cite{vig_kirillov}. Taking $A=k$, Theorem~\ref{kirillovintro} gives the existence of Kirillov models in the mod-$\ell$ setting for essentially generic objects. It follows that any Kirillov function $f:P_n\to \Ql$ that is $\Zl$-valued extends uniquely to a $\Zl$-valued Whittaker function $GL_n(F)\to \Ql$, which shows that Kirillov models define integral structures for essentially generic objects (Corollary~\ref{integralstructure}).

Our proof of the Kirillov model relies on the Rankin--Selberg functional equation of Jacquet--Piatetski-Shapiro--Shalika. Thus, for our purposes, we need to establish the functional equation for \emph{Whittaker type} families: admissible, finitely generated $A[G_n]$-modules with $V_{N_n,\psi}$ free of rank one over $A$. Our second main result, Corollary~\ref{corollary functional equation}, is the functional equation in this generality. 

In the second author's thesis \cite{Moss16.1}, the functional equation was established for ``co-Whittaker'' families using Helm's theory of the integral Bernstein center \cite{h_whitt}. Co-Whittaker families are the subset of Whittaker type families whose generic subquotients are always \emph{quotients}; this means their associated Whittaker spaces do not see much of their structure (for example, if $A$ is a field, co-Whittaker $A[G_n]$-modules have \emph{irreducible} Whittaker spaces). Thus the functional equation appearing in Corollary~\ref{corollary functional equation} is stronger than that in \cite{Moss16.1}. It is obtained by showing a certain module of bilinear forms is free of rank one (Theorem~\ref{theorem multiplicity 1}), which is the most natural statement one could hope for in this setting. By specialization, one immediately recovers the local functional equations in \cite{Tate,MR0342495,JPS1,JPS2,Moss16.2,KM17,Moss16.1}. Our proof of the functional equation uses the idea of \cite{KM17}, which reinterprets the patterns of the original proof in \cite{JPS2} in the functorial framework of Bernstein and Zelevinsky in \cite{BZ77}. {However one of the main difficulties that we have to overcome for the general type of rings $A$ we consider is that $A[X^{\pm 1}]$ need not be principal, which is a property that
simplifies the situation in the original paper \cite{JPS2} and in the work \cite{KM17} where $A$ is a field.} Not only does the method adapt to our more general setting of families of Whittaker type, it avoids the co-Whittaker machinery used in \cite{Moss16.1}.

\section*{Acknowledgements}
The first author thanks the Abdus Salam School of 
Mathematical Sciences where parts of the paper were written for its hospitality, and the CNRS for giving him a ``délégation.''  The second author thanks the CNRS for its ``poste-rouge'' visitor funding, and the laboratoire de mathématiques de Versailles, where the paper started, for their hospitality, as well as the University of Utah. Both authors thank Olivier Fouquet and Rob Kurinczuk for their interest, and a referee for their useful comments and corrections.

\section{Setting}
We denote by $A$ a commutative Noetherian $W(k)$-algebra for $k$ an algebraically closed field of characteristic $\ell\neq p$. If $G$ is a locally pro-$p$ group, we denote by 
$\Rep_A(G)$ the category of smooth $A[G]$-modules. We will work inside $\Rep_A(G)$, hence we will not
mention that the modules 
that we consider are smooth anymore. We denote by $q$ the cardinality of the residue field of $F$, and fix a square root $q^{1/2}$ of $q$ in $W(k)$ which we use to define the square root of the modulus character $\d_G$ of $G$ (see \cite[2.6]{Vbook}). We then use it to normalize the different induction functors. For $V$ and $V'$ in $\Rep_A(G)$ and $\chi:G\rightarrow A^\times$ a character of $G$, we denote by 
\[\Bil_{G}(V,V',\chi)\] the set of bilinear forms from $V\times V'$ to $A$ such that \[B(g.\ ,g.\ )=\chi(g)B.\]
If $H$ is a closed subgroup of $G$, We denote by $\ind_H^G$ and $\Ind_H^G$ respectively normalized compact 
induction and normalized induction from $\Rep_A(H)$ to $\Rep_A(G)$. If $R$ is any commutative $A$-algebra and $V\in \Rep_A(G)$, we set $V_R:=V\otimes_A R\in \Rep_R(G)$ and we shall 
use both notations. 
We denote by $F$ a non archimedean local field, and by $\val:F\rightarrow \Z$ its normalized valuation. 
We denote by $G_n$ the group $\GL_n(F)$ and by $P_n$ its mirabolic subgroup (consisting of matrices with last row equal to 
$(0,\dots,0,1)$). We denote by $N_n$ the subgroup of $G_n$ the elements of which are upper triangular unipotent matrices. We occasionally drop the subscript $(-)_n$ on these groups when it is not needed. We consider a 
fixed nontrivial $W(k)^\times$-valued additive 
character $\psi$ of $F$. We set $\psi_A:F \overset{\psi}{\rightarrow} W(k)^\times \rightarrow A^\times$ where 
the arrow on the right comes from the $W(k)$-algebra structure on $A$. By a common abuse of notation we also denote 
$\psi_A:N_n\rightarrow A^\times$ the character defined by the formula 
\[\psi_A(n)=\psi_A(\sum_{k=1}^{n-1} n_{k,k+1}).\]
We will use the following basic lemma.

\begin{lemma}\label{induction and tensor product}
Let $G$ be a locally compact totally disconnected group such that $G$ has a compact open subgroup 
with pro-order invertible in $A$, $H$ a closed subgroup of $G$, and $V\in \Rep_A(H)$. If $R$ is a commutative $A$-algebra, 
then \[\ind_{H}^{G}(V\otimes_A R)\simeq \ind_{H}^{G}(V)\otimes_A R\]

\end{lemma}
\begin{proof} 
For this proof only, we consider non normalized induction, for lighter notations. We denote by \[b:\ind_{H}^{G}(V)\otimes_A R\rightarrow \ind_{H}^{G}(V\otimes_A R)\] the obvious map. For any compact open subgroup $K$ 
of $G$, it sends $\ind_{H}^{G}(V)^K\otimes_A R$ to $\ind_{H}^{G}(V\otimes_A R)^K$, and we denote by $b_K$ its restriction to 
$\ind_{H}^{G}(V)^K\otimes_A R$. Let $K_0$ be a compact open subgroup of $G$ with pro-order invertible in $K$, and consider only such groups $K$ inside $K_0$. 
Set $(g_i)_i$ a set of representatives of $H\backslash G /K$. For $w\in V\otimes_A R$ fixed by $H\cap g_iKg_i^{-1}$ we set 
$f_{g_i,w}^K$ the only function in $\ind_{H}^{G}(V\otimes_A R)^K$ supported on $Hg_iK$ such that $f(g_i)=w$. Then any function in $\ind_{H}^{G}(V\otimes_A R)^K$ can be written
uniquely (i.e. with the $w_i$ unique)
as a finite sum \[f=\sum_i f_{g_i,w_i}^K\] and a similar statement is true for $\ind_{H}^{G}(V)^K$.
Note that $(V\otimes_A R)^{K_H}=V^{K_H}\otimes_A R$ for $K_H$ any small enough compact open subgroup of $H$ because averaging over $K_H$ is possible. 
Now each $w_i$ above is of the form $\sum_{k_i} v_{k_i}\otimes_A r_{k_i}$ with $v_{k_i}\in V^{H\cap g_iKg_i^{-1}}$. 
We set \[b'_K(f)=\sum_i b'_K(f_{g_i,w_i}^K)\] where \[b'_K(f_{g_i,w_i}^K)=\sum_{k_i} f_{g_i,v_{k_i}}^K \otimes_A r_{k_i} .\] 
One checks that $b'_K$ is well-defined, and provides an inverse to $b_K$. 
Then the maps $b'_K$ when $K$ varies form a compatible system, so that there is 
$b':\ind_{H}^{G}(V\otimes_A R)\rightarrow \ind_{H}^{G}(V)\otimes_A R$ such that $b'_{|K}=b'_K$. In particular $b'=b^{-1}$.
\end{proof}

We embed the group $G_{n-1}$ into $P_n$ via the map $g\mapsto \diag(g,1)$, so that $P_n$ becomes the semi-direct product of 
$G_{n-1}$ with its unipotent radical $U_n$. We define the derivative functors $\Phi^{\pm}$ and $\Psi^{\pm}$ as in \cite{BZ77} with respect to $\psi_A$:

\begin{itemize}

 \item The functor $\Phi^{-}:\Rep_A(P_n)\rightarrow \Rep_A(P_{n-1})$. For $(\pi,V)\in \Rep_A(P_n)$, one sets 
\[\Phi^{-} V =V_{U_n,\psi_A}=V/\langle \pi(u)v-\psi_A(u)v,\ (u,v)\in U_n\times V \rangle\] and $P_{n-1}$ acts on $\Phi^{-}V$ by
 $h.\overline{v}= \overline{\delta_{U_n} (h)^{-1/2}\pi (\diag(h,1))v}$, for $h\in P_{n-1}$.

\item The functor $\Phi^{+}:\Rep_A(P_{n-1})\rightarrow \Rep_A(P_n)$, defined by 
\[\Phi^{+} \pi = \ind_{P_{n-1}U_n}^{P_n}(\pi \otimes \psi_A).\]

\item The functor $\Psi^{-}:\Rep_A(P_n)\rightarrow \Rep_A(G_{n-1})$. For $(\pi,V)\in \Rep_A(P_n)$, one sets 
\[\Psi^{-} V =V_{U_n}=V/\langle \pi(u)v-v,\ (u,v)\in U_n\times V \rangle\] and $G_{n-1}$ acts on $\Psi^{-}V$ by
 $g.\overline{v}= \overline{\delta_{U_n} (g)^{-1/2}\pi (\diag(g,1))v}$ for $g\in G_{n-1}$.

\item The functor $\Psi^{+}:\Rep_A(G_{n-1})\rightarrow \Rep_A(P_n)$ defined by 
\[\Psi^{+} \pi = \ind_{G_{n-1}U_n}^{P_n}(\pi \otimes \psi_A)= \d_{U_n}^{1/2}\pi \otimes \psi_A.\]

\end{itemize}
The $i$'th Bernstein--Zelevinsky derivative functor $(-)^{(i)}:\Rep_A(P_n)\rightarrow \Rep_A(G_{n-i})$ is defined as $$M^{(i)}= \Psi^-(\Phi^-)^{i-1}M$$ for $M\in \Rep_A(P_n)$.

\begin{rem}
Note that in our context (as opposed to the special case $A=k$), there is no reason for the set of characters $F\to A^{\times}$ to be conjugate under the action of the diagonal torus of $G_n$. For example if $A = k\times k$, one cannot conjugate one factor to the other. However our character $\psi_A$ is ``rigid'' in the sense that it is obtained by extending scalars along the fixed structure morphism $W(k)\to A$. All nondegenerate \emph{rigid} characters $F\to A^{\times}$ are conjugate. In order to lighten notation, we omit any reference to the $W(k)$-algebra structure of $A$ while writing $\psi_A$ and the Bernstein-Zelevinsky functors.

We inevitably have to consider derivatives with respect to $(\psi_A)^{-1} = (\psi^{-1})_A$. Since $\psi$ and $\psi^{-1}$ are in the same orbit under the action of the diagonal torus, the same is true for $\psi_A$ and $\psi_A^{-1}$. 
\end{rem}

Amongst the properties satisfied by the Bernstein-Zelevinsky functors, hereunder are those that we shall need.

\begin{prop}\label{proposition properties of derivatives}
\begin{enumerate}[1)]
\item The Bernstein-Zelevinsky filtration: if $M$ is an $A[P_n]$-module then 
$\{0\}\subset M_n \subset \dots \subset M_1 =M$ with $M_i\simeq (\Phi^+)^{i-1}(\Phi^-)^{i-1} M$ and 
$M_{i}/M_{i+1}\simeq (\Phi^+)^{i-1}\Psi^+(M^{(i)})$.
\item If $V$ and $V'$ are $A[G_n]$-modules, $\chi:G_{n+1}\rightarrow A^\times$ is a character, and 
$M$ and $M'$ are smooth $A[P_n]$-representations, then 
\[\Bil_{P_{n+1}}(\Psi^+V,\Psi^+V',\chi)\simeq \Bil_{G_n}(V, V',\chi)\]
\[\Bil_{P_{n+1}}(\Phi^+M,\Phi^+M',\chi)\simeq \Bil_{P_n}(M,M',\chi)\]
\[\Bil_{P_{n+1}}(\Psi^+V,\Phi^+M,\chi)=\{0\}\]
\item Let $R$ be a commutative $A$-algebra. If $M$ is an $A[P_n]$-module and $V$ is an $A[G_n]$-module, then 
$\Phi^{\pm}(M\otimes_A R)\simeq \Phi^{\pm}(M)\otimes_A R$, 
 $\Psi^{-}(M\otimes_A R)\simeq \Psi^{-}(M)\otimes_A R$ and $\Psi^{+}(V\otimes_A R)\simeq \Psi^{+}(V)\otimes_A R$.
\end{enumerate}
\end{prop}
\begin{proof}
For the part 1), the proof in \cite{BZ77} for $\C[P_n]$-modules holds without modification for $W(k)[P_n]$-modules. Now if one starts with an $A[P_n]$-module, 
hence a $W(k)[P_n]$-module, the filtration by $W(k)[P_n]$-modules is in fact a filtration by $A[P_n]$-modules. Indeed the functors $\Phi^{\pm}$ and $\Psi^{\pm}$ factor through 
the forgetful functor which takes an $A[G]$-module to a $W(k)[G]$-module for $G=P_n$ or $G_{n-1}$.
For part 2), the proofs in \cite{BZ77} hold over $A$ directly. 
Finally, Part 3) is obvious for $\Psi^-$ and $\Phi^-$ and it is a consequence of Lemma \ref{induction and tensor product} for $\Phi^+$ and $\Psi^+$.
\end{proof}

Finally we end this section with the following definition which is central for our purpose. 

\begin{definition} We say that an $A[G_n]$-module $V$ is of Whittaker type if it is admissible, is $A[G_n]$-finitely generated, and if $V^{(n)}$ is isomorphic 
to $A$ as an $A$-module.
\end{definition}

If $V$ is of Whittaker type, then $\Hom_A(V^{(n)},A)\simeq \Hom_{N_n}(V,\psi_A)\simeq 
\Hom_{G_n}(V,\Ind_{N_n}^{G_n}(\psi_A))$ is a free $A$-module of rank one. We denote by $\Wh(V,\psi_A)$ the Whittaker {space} of 
$V$ with respect to $\psi_A$, i.e. the image of $V$ in $\Ind_{N_n}^{G_n}(\psi_A)$ via any homomorphism corresponding to an isomorphism between $V^{(n)}$ and $A$. Observe that in this case the map $V\rightarrow V^{(n)}\simeq A$ factors as \[V\rightarrow \Wh(V,\psi_A)\overset{\mathrm{Ev}_{I_n}}{\rightarrow} A,\] where $\mathrm{Ev}_{I_n}$ denotes the function given by evaluating at the identity. This implies that $\Wh(V,\psi_A)^{(n)} \simeq V^{(n)} \simeq A$. We will use the following observation. Set $w_n=\begin{pmatrix} & & 1\\ & \iddots & \\ 1 & & \end{pmatrix}\in G_n$. If $V$ is an $A[G_n]$-module of Whittaker type, then the $A[G_n]$-module 
$\widetilde{V}$ with underlying $A$-module $V$, but with $G_n$-action twisted by $i:g\mapsto w_n {}^t g^{-1} w_n^{-1}$ is also of Whittaker type. Indeed admissibility and finite generation are obvious, and because $N_n$ is $i$-stable and $\psi(i(n))=\psi^{-1}(n)$ for $n\in N_n$, we have $\widetilde{V}^{(n)}\simeq V^{(n)}$. Moreover if for $W\in \Wh(V,\psi_A)$ we set \[\widetilde{W}(g)=W(w_n {}^t g^{-1})\] for $g\in G_n$, then the map $W\mapsto \widetilde{W}$ is an $A$-module isomorphism from $\Wh(V,\psi_A)$ to $\Wh(\widetilde{V},\psi_A^{-1})$.

\section{The Rankin-Selberg functional equation for representations of Whittaker type}

\subsection{Schur's lemma for Whittaker spaces}




Given a representation $V$ of such that $V^{(n)}\cong A$, its so-called ``Schwartz functions'' are the sub $A[P_n]$-module which is the image of the embedding 
$$(\Phi^+)^{n-1}\Psi^+(\textbf{1}_A)\to V.$$ The proof of the following proposition gives a hint as to their importance in this circle of ideas.

\begin{prop}\label{prop:whittakertypeschur}
Let $A$ be a Noetherian $W(k)$-algebra and let $V$ be an $A[G]$-module of Whittaker type. Then the natural map
$$A\to \End_{A[G]}(\Wh(V,\psi_A))$$ is an isomorphism (i.e. the Whittaker space of $V$ satisfies Schur's lemma).
\end{prop}
\begin{proof}
We identify $\Wh(V,\psi_A)^{(n)} = V^{(n)}=A$. Let $\mathcal{S}(V,\psi_A)$ denote the Schwartz functions of $\Wh(V,\psi_A)$. Let $\phi$ be a $G$-equivariant endomorphism of $\Wh(V,\psi_A)$. Since $\phi$ is $P$-equivariant, it follows that $\phi(\mathcal{S}(V,\psi_A))\subset \mathcal{S}(V,\psi_A)$, this is a consequence of the Bernstein-Zelevinsky filtration as the only possible nonzero derivative of $\phi(\mathcal{S}(V,\psi_A))$ must be $\phi(\mathcal{S}(V,\psi_A))^{(n)}$ as it is the case for $\mathcal{S}(V,\psi_A)$. Let $\widetilde{\mathcal{S}}(V,\psi_A)$ denote the $A[G]$-submodule of $\Wh(V,\psi_A)$ generated by $\mathcal{S}(V,\psi_A)$. Thus the following map is well defined,
\begin{align*}
\End_G(\Wh(V,\psi_A)) &\to \End_G(\widetilde{\mathcal{S}}(V,\psi_A))\\
\phi&\mapsto \phi|_{\widetilde{\mathcal{S}}(V,\psi_A)}.
\end{align*}
Suppose $\phi|_{\widetilde{\mathcal{S}}(V,\psi_A)}\equiv 0$. Then $\phi$ defines a morphism from the quotient $Q:=\Wh(V,\psi_A)/\widetilde{\mathcal{S}}(V,\psi_A)$ to $\Wh(V,\psi_A)$. But every nonzero submodule in the target space $\Wh(V,\psi_A)$ has nonzero highest derivative, so the fact that $Q^{(n)}=0$ (by exactness of the derivative functors) implies $\phi(Q)=0$. Therefore $\phi\equiv 0$, so the restriction map is injective.

On the other hand, since $\mathcal{S}(V,\psi_A)$ generates $\widetilde{\mathcal{S}}(V,\psi_A)$, the map $\phi\mapsto \phi|_{\mathcal{S}(V,\psi_A)}$ defines an injection $$\End_G(\widetilde{\mathcal{S}}(V,\psi_A))\to \End_P(\mathcal{S}(V,\psi_A)).$$ But $\mathcal{S}(V,\psi_A)$ satisfies Schur's lemma, by the adjunctions of $\Phi^+$ and $\Phi^-$ (e.g. \cite[3.1.16]{EH14}). Thus the composition
$$A\to \End_G(\Wh(V,\psi_A))\hookrightarrow \End_G(\widetilde{\mathcal{S}}(V,\psi_A))\hookrightarrow \End_G(\mathcal{S}(V,\psi_A))$$ is an isomorphism. Thus the maps are all isomorphisms.
\end{proof}

We note the following corollary of the previous proof, which is a sort of generalization of \cite[3.2.3 Lemma]{EH14}.

\begin{corollary}
If $V$ is an $A[G]$-module of Whittaker type, and $\mathcal{S}(V,\psi_A)$ denotes the Schwartz functions of $\Wh(V,\psi_A)$, then restriction to $\mathcal{S}(V,\psi_A)$ defines an isomorphism $$\End_G(\Wh(V,\psi_A)) \to \End_P(\mathcal{S}(V,\psi_A)).$$
\end{corollary}

\subsection{Bilinear forms}

We let $A'$ be a commutative finitely generated $W(k)$-algebra. As in \cite{Moss16.1}, \cite{Moss16.2}, we introduce $S$ the multiplicative subset of $(A\otimes_{W(k)} A')[X^{\pm1}]$ given by the Laurent polynomials with dominant and trailing coefficient equal 
to $1$, we set \[R= S^{-1} \left((A\otimes_{W(k)} A')[X^{\pm 1}]\right).\] 
We consider invariant measures with values in $R$ as in \cite{KM17}. Note that 
$R$ is a Noetherian $W(k)$-algebra. For $r\in R^\times$, we write \[\chi_r:\GL_n(F)\rightarrow  R^\times\] for 
the composition of the unramified character of $F^\times$ sending a uniformizer of $F$ to $r^{-1}$ with the determinant on 
$\GL_n(F)$. For $m<n$ two positive integers we set 
\[U_{m+1,n-m-1}=\{\begin{pmatrix} I_{m+1} & x \\ & y \end{pmatrix},\ x \in \mathcal{M}_{m+1,n-m-1},\ y\in N_{n-m-1}\}\]
and consider $\psi$ as a character of $U_{m+1,n-m-1}$ through the projection on the $y$ coordinate. We consider 
$V'$ an $A'[G_m]$-module of Whittaker type and see it as an $A'[G_m U_{m+1,n-m-1}]$-module by making 
$U_{m+1,n-m-1}$ act trivially. We now follow the original method of \cite{JPS2}, as adapted in \cite{KM17}. Note that \cite{KM17} contains a typo throughout the whole section 3.2: $V$ and $V'$ should be replaced everywhere by their Whittaker models.

 Before stating the functional equation, we do some reduction steps backwards. For $X\subset G_m$, we set \[X^{(l)}=\{g\in X,\ \mathrm{val}(\det(g))=l\}.\] We recall the following (see \cite[proposition 2.3.6]{JPS1}, however in our context where 
 $A$ is not necessarily a domain we have to be a little careful so we give the proof): 
 
\begin{lemma}\label{lemma support}
For $W\in \Ind_{N_n}^{G_n}(\psi_A)$ and fix $l\in \Z$. Then $W_{|G_{n-1}^{(l)}}$ is compactly supported modulo $N_{n-1}$.
\end{lemma}
\begin{proof}
Set $A_n$ the diagonal torus of $G_n$. By Iwasawa decomposition it is sufficient to check that $W_{|A_{n-1}^{(l)}}$ is compactly supported. However because of the restriction on the determinant in $A_{n-1}^{(l)}$ it is sufficient to prove that there is $r\in \Z$ such that the condition $\val(a_i)-\val(a_{i+1})\leq r$ for $i=1,\dots,n-2$ implies that $W(\diag(a_1,\dots,a_{n-1},1))$ is zero. Consider \[u_i(x)=\diag(1,\dots,1,\begin{pmatrix} 1 & x \\ & 1\end{pmatrix},1,\dots, 1)\in G_{n-1}\] where the $2\times 2$ matrix is in position $i$, then $W$ is fixed by $u_i(x)$ on the right for $x$ nonzero but small enough, but on the other hand, denoting by $\rho$ the action by right translation, it is equal to \[(\rho(u_i(x))W-W)(\diag(a_1,\dots,a_{n-1},1))=(\psi(\frac{a_ix}{a_{i+1}})-1)W(\diag(a_1,\dots,a_{n-1},1)).\] Now 
$\psi(\frac{a_ix}{a_{i+1}})-1$ belongs to $W(k)^\times$ for $\val(a_i)-\val(a_{i+1})$ small enough, hence is invertible in $A$ and the result follows.
\end{proof}
 
For $W\in (\Phi^+)^{m-1}\Psi^+(\1_R)=\ind_{N_m}^{P_m}(\psi_R)$ and 
$W' \in \Wh(V',\psi_{A'}^{-1})_R$ we set 
\[I_1(X,W,W')=\sum_{l\in \mathbb{Z}} [\int_{N_m\backslash P_m^{(l)}} W(p) W'(p)dg]q^{l/2}X^l=\int_{N_m\backslash P_m} W(p)  W'(p)  \chi_{q^{-1/2}X^{-1}}(p)dp.\] As in \cite{Moss16.1}, one has $I_1(X,W,W')\in R$.

\begin{lemma}\label{lemma reduction 1}
Suppose $V'$ is an $A'[G_m]$-module of Whittaker type. \[\Bil_{P_m}((\Phi^+)^{m-1}\Psi^+(\1_R), \Wh(V',\psi_{A'}^{-1})_R,\chi_{q^{-1/2}X})\simeq \Bil_{P_m}((\Phi^+)^{m-1}\Psi^+(\1_R),(\Phi^+)^{m-1}\Psi^+(\1_R),\chi_{q^{-1/2}X})\]
\[\simeq \Bil_{G_0}(\1_R,\1_R,\chi_{q^{-1/2}X})\simeq R.\] 
\end{lemma}
\begin{proof}
The last isomorphism is clear, and the one before follows from Proposition \ref{proposition properties of derivatives}. For the first isomorphism, we certainly have a natural $R$-linear map 
\[\mathrm{Res}_0: B\mapsto B_{|(\Phi^+)^{m-1}\Psi^+(\1_R)\times (\Phi^+)^{m-1}\Psi^+(\1_R)}\] from \[\Bil_{P_m}((\Phi^+)^{m-1}\Psi^+(\1_R), \Wh(V',\psi_{A'}^{-1})_R,\chi_{q^{-1/2}X})\] to \[\Bil_{P_m}((\Phi^+)^{m-1}\Psi^+(\1_R) ,(\Phi^+)^{m-1}\Psi^+(\1_R),\chi_{q^{-1/2}X}).\] The map is moreover clearly surjective because the natural basis $\Psi_0(X, \ . \ , \ . \ )$ defined by 
\[\Psi_0(X, W, W')=\int_{N_m\backslash P_m} W(p)  W'(p)  \chi_{q^{-1/2}X^{-1}}(p)dp\] of the $R$-module on the right is the image of $I_1(X, \ . \ , \ . \ )$. Let's prove that it is injective. 
Take an element $B$ of $\Bil_{P_m}((\Phi^+)^{m-1}\Psi^+(\1_R), \Wh(V',\psi_{A'}^{-1})_R,\chi_{q^{-1/2}X})$ which vanishes on 
$(\Phi^+)^{m-1}\Psi^+(\1_R)\times (\Phi^+)^{m-1}\Psi^+(\1_R)$. Consider the Bernstein-Zelevinsky filtration of $M=\Wh(V',\psi_{A'}^{-1})_R$. By commutation of derivatives and tensor product it has bottom piece $(\Phi^+)^{m-1}\Psi^+(\1_R)$ and the other sub-quotients are of the form 
$(\Phi^+)^{i-1}\Psi^+(M^{(i)})$ with $i<m$. But \[\Bil_{P_m}((\Phi^+)^{m-1}\Psi^+(\1_R), (\Phi^+)^{i-1}\Psi^+(M^{(i)}),\chi_{q^{-1/2}X})\]
\[\simeq \Bil_{P_m}((\Phi^+)^{m-1}\Psi^+(\1_R), (\Phi^+)^{i-1}\Psi^+(M^{(i)}),\chi_{q^{-1/2}X})
=\{0\}\] thanks to Proposition \ref{proposition properties of derivatives} again, hence $B=0$ and this finishes the proof.
\end{proof}

For $W\in (\Phi^+)^{m}\Psi^+(\1_R)=\ind_{N_m}^{P_{m+1}}(\psi_R)$ and $W'\in \Wh(V',\psi_{A'}^{-1})_R$ we write 
\[I_2(X, W,W')=\int_{N_m\backslash G_m} W(g)  W'(g) \chi_{q^{1/2}X^{-1}}(g)dg\] 
\[=\int_{P_m\backslash G_m}(\int_{N_m\backslash P_m} W(pg) W'(pg)  \chi_{q^{-1/2}X^{-1}}(p)dp )\chi_{q^{1/2}X^{-1}}(g)dg\]
\[= \int_{P_m\backslash G_m} I_1(X,\rho(g)W,\rho(g)W')\chi_{q^{1/2}X^{-1}}(g)dg .\]
I.e. \begin{equation}\label{equation compatibility psi1 psi2} I_2(X, W,W')=  \int_{P_m\backslash G_m} I_1(X,\rho(g)W,\rho(g)W')
\chi_{q^{1/2}X^{-1}}(g)dg \end{equation}

\begin{lemma}\label{lemme reduction 2}
Let $N$ be an $R[P_m]$-module, then 
\[\Bil_{G_m}(\Phi^+ N , \Wh(V',\psi_{A'}^{-1})_R ,\chi_{q^{1/2}X})=\Bil_{P_m}(N,\Wh(V',\psi_{A'}^{-1})_R ,\chi_{X}),\]
in particular 
\[\Bil_{G_m}((\Phi^+)^{m}\Psi^+(\1_R), \Wh(V',\psi_{A'}^{-1})_R),\chi_{q^{1/2}X})\simeq 
\Bil_{P_m}((\Phi^+)^{m-1}\Psi^+(\1_R), \Wh(V',\psi_{A'}^{-1})_R ,\chi_{X}).\] Moreover the isomorphism from the right to the left sends the basis 
$I_2(X,\ , \ )$ to $I_1(X,\ , \ )$.
\end{lemma}
\begin{proof}
 By Mackey theory $\Phi^+(N)_{|G_m}\simeq  \ind_{P_m}^{G_m}(N)$ as explained 
in the proof of \cite[Lemma 3.8]{KM17}. Hence 
\[\Bil_{G_m}(\Phi^+(N), \Wh(V',\psi_{A'}^{-1})_R,\chi_{q^{1/2}X})\simeq 
\Bil_{G_m}(\ind_{P_m}^{G_m}(N), \Wh(V',\psi_{A'}^{-1})_R ,\chi_{q^{1/2}X}).\]
By \cite[Proposition 2.29]{BZ76} this latter space is isomorphic to $\Bil_{P_m}(N, \Wh(V',\psi_{A'}^{-1})_R ,\chi_X)$ 
and this concludes the proof of the first assertion. The proof of the second assertion follows from taking $N=(\Phi^+)^{m-1}\Psi^+(\1_R)$.
That of the last assertion follows from the explicit description of $\alpha$ 
in \cite[Proposition 2.29]{BZ76} and Equation (\ref{equation compatibility psi1 psi2}).
\end{proof}

We take a break in our chain of isomorphisms to record the following two key lemmas.
\begin{lemma}\label{finitelengthbilzero}
Let $N$ be an $A[G_l]$-module and $N'$ an $A'[G_l]$-module such that $N$ and $N'$ are admissible and $G_l$-finitely generated. Then $\Bil_{G_l}(N_R,N'_R,\chi_{X})=\{0\}$.
\end{lemma}
\begin{proof}
In this proof, we set $V:=N_R\otimes_R N'_R$. Note that $A\otimes_{W(k)} A'$ is Noetherian  since $A$ is Noetherian, and $A'$ is finitely generated as a $W(k)$-algebra, hence $R$ is Noetherian as well. The module $V$ is admissible and finitely generated as an $R[G_l\times G_l]$-module. We would like to prove that
$\Hom_{R[G_l]}(V, \chi_X) = \{0\}$ for the diagonal action of $G_l$. Choose a compact open subgroup $K$ of $G_l\times G_l$ such that $V$ is generated over $R[G_l\times G_l]$ by $V^K$. The center $Z=F^{\times}$ acts on $V^K$ via the diagonal embedding $Z\hookrightarrow G_l \times G_l$. Let $\varpi$ in $\End_R(V^K)$ denote the action of a uniformizer $\varpi_F$ of $F$. The integer powers of $\varpi$ generate a sub-$R$-algebra $\langle \varpi\rangle$ of $\End_R(V^K)$. By admissibility, $V^K$ is finitely generated as an $R$-module, hence $\End_R(V^K)$ is finitely generated as an $R$-module because $R$ is Noetherian. Again because $R$ is Noetherian the sub-$R$-module $\langle\varpi\rangle$ is also finitely generated over $R$. Hence by Cayley-Hamilton Theorem the endomorphism $\varpi$ satisfies a monic polynomial $f(X)$ in $R[X]$. Moreover, since $\varpi$ acts invertibly on $V^K$, the constant term of $f(X)$ is a unit in $R$. Since we have $f(\varpi)v = 0$ for any $v$ in $V^K$, we also have $f(\varpi)v=0$ for all $v$ in $V$, since $V^K$ generates $V$ and $\varpi$ commutes with the action of $G_l\times G_l$.

Now let $\phi$ be any element of $\Hom_{R[G_l]}(V,\chi_X)$. Then $$0 = \phi(f(\varpi)v) = f(\chi_X(\varpi_F))\phi(v) = f(X^{-l})\phi(v).$$ Since $f(X^l)$ is a monic polynomial with unit constant term, $f(X^{-l})$ is invertible in the ring $R$. We conclude that $\phi(v)=0$.
\end{proof}

\begin{lemma}\label{lemma preparation to reduction 3}
Let $N$ be an admissible, finitely generated $A[G_l]$-module for $1\leq l \leq m$, then 
\[\Bil_{G_m}({(\Phi^+)}^{m-l}\Psi^+(N_R), \Wh(V',\psi_{A'}^{-1})_R,\chi_{q^{1/2}X})=\{0\}.\]
\end{lemma}
\begin{proof}
If $l=m$, then 
\[\Bil_{G_m}(\Psi^+(N_R), \Wh(V',\psi_{A'}^{-1})_R,\chi_{q^{1/2}X})=
\Bil_{G_m}(N_R, \Wh(V',\psi_{A'}^{-1})_R,\chi_{X})\] by definition of $\Psi^+$, which is equal to zero by Lemma~\ref{finitelengthbilzero}.

If $l<m$, then \[\Bil_{G_m}({(\Phi^+)}^{m-l}\Psi^+(N_R), \Wh(V',\psi_{A'}^{-1})_R,\chi_{q^{1/2}X})\]
\[\simeq 
\Bil_{P_m}({(\Phi^+)}^{m-l-1}\Psi^+(N_R),  \Wh(V',\psi_{A'}^{-1})_R,\chi_{X}) \] 
according to Lemma \ref{lemme reduction 2}. Now thanks to the Bernstein-Zelevinsky filtration, it is sufficient to prove (as the  Bernstein-Zelevinsky functors commute with 
tensor product) that 
\[\Bil_{P_m}({(\Phi^+)}^{m-l-1}\Psi^+(N_R), {(\Phi^+)}^{m-l'-1}\Psi^+(N'_R),
\chi_{X}\otimes \psi)=\{0\}\] when $N'$ is an admissible finite type $A'[G_{l'}]$-module for $0\leq l' \leq m$. However 
if $l\neq l'$ (in particular when $l'=0$ this is automatic) this follows at once from the adjunction properties of the 
Bernstein-Zelevinsky functors. If $l=l'$ these properties tell us that 
\[\Bil_{P_m}({(\Phi^+)}^{m-l-1}\Psi^+(N_R), {(\Phi^+)}^{m-l-1}\Psi^+(N'_R),
\chi_{X})\simeq \Bil_{G_l}(N_R, N'_R,
\chi_{X}).\] This latter space is zero by Lemma~\ref{finitelengthbilzero}.
\end{proof}

Here is the final step. Recall that $V$ is an $A[G_n]$-module whereas $V'$ is an $A'[G_m]$-module. For $W\in \Wh(V,\psi_A)_R$ and $W'\in \Wh(V',\psi_{A'}^{-1})_R$ we set 
\[I(X,W,W')=\sum_{l\in \Z} (\int_{N_m\backslash G_m^{(l)}} W(g) W'(g)dg) q^{l(n-m)/2} X^l.\] By \cite{Moss16.1} again, one has $I(X,W,W')\in R$

\begin{lemma}\label{lemma reduction 3}
The restriction map $B\mapsto B|_{(\Phi^+)^{m}\Psi^+(\textbf{1}_R)\times \Wh(V',\psi_{A'}^{-1})_R}$ gives an isomorphism between the following two $R$-modules:
\[\Bil_{G_m}((\Phi^-)^{n-m-1}\Wh(V,\psi_A)_R, \Wh(V',\psi_{A'}^{-1})_R,\chi_{q^{1/2}X})\]
\[\simeq \Bil_{G_m}((\Phi^+)^{m}\Psi^+(\1_R), \Wh(V',\psi_{A'}^{-1})_R,\chi_{q^{1/2}X}).\]
Moreover 
\[\Bil_{G_m}((\Phi^-)^{n-m-1}\Wh(V,\psi_A)_R, \Wh(V',\psi_{A'}^{-1})_R,\chi_{q^{1/2}X})\] is canonically isomorphic to 
\[\Bil_{G_m U_{m+1,n-m-1}}(\Wh(V,\psi_A)_R, \Wh(V',\psi_{A'}^{-1})_R,\chi_{q^{(n-m)/2}X}\otimes \psi)\] as an $R$-module. 
\end{lemma}
\begin{proof}
First note that $I(X,\ ,\ )$ belongs to 
\[\Bil_{G_m U_{m+1,n-m-1}}(\Wh(V,\psi_A)_R, \Wh(V',\psi_{A'}^{-1})_R,\chi_{q^{(n-m)/2}X}\otimes \psi).\] 
By definition of $(\Phi^-)^{n-m-1}$ the space 
\[\Bil_{G_m U_{m+1,n-m-1}}(\Wh(V,\psi_A)_R, \Wh(V',\psi_{A'}^{-1})_R,\chi_{q^{(n-m)/2}X}\otimes \psi)\] is canonically isomorphic to 
\[\Bil_{G_m}((\Phi^-)^{n-m-1} \Wh(V,\psi_A)_R, \Wh(V',\psi_{A'}^{-1})_R,\chi_{q^{1/2}X}).\] In particular we see $I(X,\ ,\ )$ as an element of this $R$-module. 
 
Now we prove that the restriction map of the statement has to be injective. Set $M:= \Wh(V,\psi_A)_R$. By the Bernstein-Zelevinsky filtration it is sufficient to prove that 
\[\Bil_{G_m}((\Phi^-)^{n-m-1}{(\Phi^+)}^{i}\Psi^+(M^{(i+1)}), \Wh(V',\psi_{A'}^{-1})_R,\chi_{q^{1/2}X})=\{0\}\] when 
$i=0,\dots,n-2$. However for $i < n-m-1$  module  $(\Phi^-)^{n-m-1}{(\Phi^+)}^{i}\Psi^+(M^{(i+1)})$ is equal to zero by the properties of the Bernstein-Zelevinsky functors, and 
for $n+1-m\leq i\leq n-2$ it is equal to ${(\Phi^+)}^{i+m+1-n}\Psi^+(M^{(i+1)})$. 
Because derivatives commute with tensor products, each $M^{(i+1)}$ is of the form $N_R$ for $N$ an admissible finite-type $A[G_{n-i-1}]$-module. But we know that
\[\Bil_{G_m}({(\Phi^+)}^{i+m+1-n}\Psi^+(N_R), \Wh(V',\psi_{A'}^{-1})_R,\chi_{q^{1/2}X})=\{0\}\] thanks to Lemma \ref{lemma preparation to reduction 3}. 

Finally the restriction of bilinear forms map is surjective as the restriction of the bilinear form $I(X,\ , \ )$ is precisely $I_2(X,\ ,\ )$, which is a basis of the target space.
\end{proof}

Together Lemmas \ref{lemma reduction 1}, \ref{lemme reduction 2}, \ref{lemma reduction 3} imply:

\begin{thm}\label{theorem multiplicity 1}
Let $A$ and $A'$ be Noetherian $W(k)$-algebras, and suppose that $A'$ is finitely generated over $W(k)$. Let $m<n$ be two positive integers, let $V$ be an $A[G_n]$-module of Whittaker type, and let $V'$ be a $A'[G_m]$-module of Whittaker type. 
Then the $R$-module \[\Bil_{G_m U_{m+1,n-m-1}}(\Wh(V,\psi_A)_R, \Wh(V',\psi_{A'}^{-1})_R,\chi_{q^{(n-m)/2}X}\otimes \psi_R)\] is a free $R$-module of rank one. 
\end{thm}
{Note that the finite generation of $A'$ is only needed insofar as it guarantees $A\otimes_{W(k)}A'$ is Noetherian, for the purpose of Lemma~\ref{finitelengthbilzero}.}

\subsection{Functional Equation}

Again let $A'$ be a finitely generated commutative $W(k)$-algebra. We now introduce some more notation. We denote by $\M_{a,b}$ the space of $a\times b$ matrices with coefficients in $F$. 
Let $V$ and $V'$ be $A[G_n]$- and $A'[G_m]$-modules of Whittaker type, with $n>m$. For $0\leq j \leq n-m-1$, $W\in \Wh(V,\psi_A)$ and $W'\in \Wh(V',\psi_{A'}^{-1})$ we set 
\[I(X,W,W';j)=\sum_{l\in \Z}c_l(W,W';j) X^l,\] 
 where \[c_l(W,W';j)= \int_{\M_{j,m}} \int_{N_m\backslash G_m^{(l)}}W\begin{pmatrix} g & & \\ x & I_j & \\ 
 & & I_{n-m-j} \end{pmatrix} \otimes W'(g) dg dx.\] 
 Note that $I(X,W,W')=I(X,W,W';0)$.

We set \[w_{t,r}=\diag(I_t,w_r)\in G_{t+r}.\]
It is explained in \cite[(2.11), p. 395]{JPS2} that the integrals $I(q^{-1}X^{-1}, \rho(w_{m,n-m})\widetilde{W}, \widetilde{W}';n-m-1)$ and $ I(X, W, W';0)$ in Corollary \ref{corollary functional equation} below all belong to $D(\Wh(V,\psi_A),\Wh(V',\psi_{A'}^{-1}))$, this is just a simple change of variable. It is then explained in \cite[(2.8), p. 395]{JPS2} which crucially relies on \cite[Theorem (4.5)]{JPS1}, that the gamma factor given in the functional equation below does not depends on $j$, the proof of the functional equation being reduced to the case $j=0$. This latter fact is not trivial at all (in particular they need to introduce auxiliary integrals and do elementary Fourier analysis on finite dimensional $F$-vector spaces, which is of course fine in our context as such a vector space is an Abelian pro-$p$-group) but follows verbatim from the aforementioned arguments. This being observed, a corollary of Theorem \ref{theorem multiplicity 1} is the functional equation of local Rankin-Selberg integrals:

\begin{corollary}\label{corollary functional equation}
Let $V$ be an $A[G_n]$-module of Whittaker type, let $V'$ be an $A'[G_m]$-module of Whittaker type with $n>m$, and let $0\leq j \leq n-m-1$. Then there is a unique $\gamma(X,V,V',\psi)\in R^{\times}$ such that for all $W\in \Wh(V,\psi_A)$ and $W'\in \Wh(V',\psi_{A'}^{-1})$, one has:
\begin{equation}\label{functionaleqn}
I(q^{-1}X^{-1}, \rho(w_{m,n-m})\widetilde{W}, \widetilde{W}';n-m-1-j)=\omega_{V'}(-1)^{m-1}\gamma(X,V,V',\psi) I(X, W, W';j),
\end{equation}
where $\omega_{V'}$ is the central character of $\Wh(V',\psi_{A'})$ (c.f. Prop~\ref{prop:whittakertypeschur}).
\end{corollary}

If $f:A\to B$ and $f':A'\to B'$ are ring homomorphisms, and $V$, $V'$ are of Whittaker type, then $V\otimes_AB$ and $V'\otimes_{A'}B'$ are also of Whittaker type, by Proposition~\ref{proposition properties of derivatives}. Let 
$$R' = \tilde{S}^{-1}(B\otimes_{W(k)} B')[X^{\pm 1}],$$ where $\tilde{S}$ is the subset of $(B\otimes B')[X^{\pm1}]$ with leading and trailing coefficients 1. Denote by $f\otimes f'$ either the map $A\otimes A'\to B\otimes B'$ or the induced map $R\to R'$, depending on the context. The factor $\gamma(X,V,V',\psi)\in R^{\times}$ is compatible with extension of scalars in the following sense.
\begin{corollary}\label{corollary specialization}
$\gamma(X,V\otimes_AB,V'\otimes_{A'}B',\psi) = (f\otimes f')(\gamma(X,V,V',\psi))$ in $(R')^{\times}$.
\end{corollary}
\begin{proof}
Apply $f\otimes f'$ to Equation~(\ref{functionaleqn}) means applying $f\otimes f'$ to the coefficients $c_l(W,W';j)$, and $c_l(\rho(w_{m,n-m})\widetilde{W}, \widetilde{W}';n-m-1-j)$. Then $f\otimes f'$ can be moved inside the integrals because they are finite sums. The Whittaker space $\Wh(V\otimes_AB,\psi_B) =\Wh(\Wh(V,\psi_A)\otimes_AB,\psi_B)$ is the image of $\Wh(V,\psi_A)$ under the pushing-forward map $\Ind_U^G\psi_A\to \Ind_U^G\psi_B$, so the result follows from uniqueness in Corollary~\ref{corollary functional equation}.
\end{proof}
Thus $\gamma(X,V,V',\psi)$ generalizes the gamma factors in \cite{Tate,MR0342495,JPS1,JPS2,Moss16.2,KM17,Moss16.1}.

\section{Kirillov models}

\subsection{Kirillov models via the functional equations}

Now we can show that the Whittaker {space} of an $A[G_n]$-module of Whittaker type admits a Kirillov model, for $A$ a commutative Noetherian $W(k)$-algebra. To this end we use the following ``completeness of Whittaker models'' statement, which follows from \cite{HMconverse}, Cor 3.6.

\begin{thm}\label{completeness}
Take $W\in \ind_{N_n}^{G_n}(\psi)$. Suppose that for any 
finite-type commutative $W(k)$-algebra $A'$, and any Whittaker type $A'[G_n]$-module $V'$, one has 
$\int_{N_n\backslash G_n} W(g)\otimes W'(g)dg=0$ for all $W'\in \Wh(V',\psi_{A'}^{-1})$, then $W=0$.
\end{thm}

\begin{rem*}
The theorem stated above is slightly weaker than what is proved in \cite{HMconverse}, which only requires $V'$ to run over so-called ``co-Whittaker'' $A'[G_n]$-modules, which are in particular Whittaker type. See also \cite[Theorem 5.1]{Moss20}. 
\end{rem*}

Denoting by $Z_n$ the center of $G_n$. From the above results and the $(n,n-1)$ functional equation, we deduce the existence of modular Kirillov models. As mentioned before this proof is a simplification of that in \cite[Proposition 7.5.1]{JPS1} for $G_3$ (which used the $(n,n-2)$-functional equation).

\begin{thm}\label{theorem modular kirillov}
Let $A$ be a Noetherian $W(k)$-algebra and let $V$ be an $A[G_n]$-module of Whittaker type ($n\geq 2$). Then the map $\Res_{P_n}:W\mapsto W_{|P_n}$ is injective on $\Wh(V,\psi_A)$.
\end{thm}
\begin{proof}
We will only make use of the $(n,n-1)$ functional equation. Suppose that $W_{|P_n}=0$ for $W\in \Wh(V,\psi_A)$. Then for any finite type commutative $W(k)$-algebra $A'$, any $A'[G_{n-1}]$-module of Whittaker type $V'$ and any $W'\in \Wh(V',\psi_{A'}^{-1})$, one has $I(X, W, W';0)=0$. From Corollary \ref{corollary functional equation} we deduce that $I(X, \widetilde{W}, \widetilde{W'};0)=0$. So $I(X, \widetilde{W}, \ .\ ;0)$ vanishes on $\Wh(V',\psi_{A'}^{-1})$ for all 
$A'[G_{n-1}]$-module of Whittaker type $V'$, when $A'$ varies through all commutative finitely generated 
$W(k)$-algebras. This implies that for fixed $l\in \Z$, the same is true for $c_l(X,\widetilde{W}, \ .\ ;0)$. Now by Theorem \ref{completeness} and Lemma \ref{lemma support} we deduce that $\widetilde{W}$ vanishes on $G_{n-1}^{(l)}$ for all $l$, hence on $G_{n-1}$, hence on $P_n=N_nG_{n-1}$. This argument of course works in the other direction, hence $W$ vanishes on $P_n$ if and only if $\widetilde{W}$ vanishes on $P_n$, i.e. if and only if $W$ vanishes on $w_n { }^t\! P_n$. In particular the space 
\[E=\{W\in \Wh(V,\psi_A), \Res_{P_n}(W)=0\}\] is clearly a sub-$A[P_n]$-module of $\Wh(V,\psi)$, but also a sub-$A[{ }^t\! P_n]$-module. Because 
$P_n$ and ${ }^t\! P_n$ generate $G_n$, we deduce that $E$ is a sub-$A[G_n]$-module of $\Wh(V,\psi_A)$. Now because the linear form 
$W\mapsto W(I_n)$ is identically zero on $E\subseteq \Ind_{U_n}^{G_n}(\psi_A)$, this implies that $E=\{0\}$. 
\end{proof}

\subsection{Integral Kirillov models}

We now answer a question of Vign\'{e}ras. We set $G:=G_n$ for $n\geq 2$, and similarly we set $P=P_n$. Let $k=\Fl$, let $\psi = \psi_{\Ql}$, and let $V$ be a $\Ql[G]$-module of Whittaker type. By Theorem~\ref{theorem modular kirillov}, $\Res_{P}$ is injective on $\Wh(V,\psi)$. We set $\mathcal{K}(V,\psi)=\Res_{P}(\Wh(V,\psi))$. 
Vignéras has shown in \cite[Prop II.4]{V2} that, if the $\Zl[G]$-module $\Wh(V,\psi)_e$ of 
Whittaker functions with values in $\Zl$ is nonzero, it forms an integral structure of $\Wh(V,\psi)$, that is, a $G$-stable free $\Zl$-module that contains a $\Ql$-basis of $\Wh(V,\psi)$. If $\mathfrak{M}$ denotes the maximal ideal of $\Zl$, we denote by $r_\ell$ the reduction modulo $\mathfrak{M}$ map on functions in $\Wh(V,\psi)_e$. The representation 
$r_\ell(\Wh(V,\psi)_e)$ is then a $k[G]$-module of Whittaker type, which is isomorphic to its Whittaker model, as it is contained in $\Ind_U^G\psi_k$. Now one can define $\mathcal{K}(V,\psi)_e$ in a similar manner, and ask if $\mathcal{K}(V,\psi)_e$ is an integral structure for $\mathcal{K}(V,\psi)$. This was asked  by Vignéras in \cite{VigGL2}. Theorem \ref{theorem modular kirillov} answers this question in the affirmative. 

\begin{corollary}\label{integralstructure}
The map $\Res_{P}$ from $\Wh(V,\psi)_e$ to $\mathcal{K}(V,\psi)_e$ is bijective, hence $\mathcal{K}(V,\psi)_e$ is an integral structure on $\mathcal{K}(V,\psi)$.
\end{corollary}
\begin{proof}
Let $\Q_\ell^{\mathrm{ur}}$ be the maximal unramified extension of $\Q_\ell$. By standard results of Vignéras (see \cite[II.4.7]{Vbook}), there is some finite extension $E$ of $\Q_\ell^{\mathrm{ur}}$ and a Whittaker type $E[G_n]$-module $V_0$ such that $V = V_0\otimes_E \Ql$. Let $\mathcal{O}_E$ denote the ring of integers of $E$. It suffices to prove that $\Res_P:\Wh(V_0,\psi_E)_e\to \mathcal{K}(V_0,\psi_E)_e$ is bijective.

Clearly $\Res_{P}$ sends $\Wh(V_0,\psi_E)_e$ to a free $\mathcal{O}_E$-submodule inside $\mathcal{K}(V_0,\psi_E)_e$ in an injective manner. Take 
$K=\Res_{P}(W)\in \mathcal{K}(V_0,\psi_E)_e$. Suppose not all values of $W$ are in $\mathcal{O}_E$. Denoting by $\w_E$ a uniformizer of $E$, this means there exists a minimal $n>0$ such that $\w_E^n W$ takes values in $O_E$, in particular $\w_E^n W$ takes at least one value in $O_E^\times$.  However 
$\Res_{P}(r_\ell(\w_E^n W))=r_\ell(\w_E^n K)=0$ because $K\in \mathcal{K}(V_0,\psi_E)_e$, hence $r_\ell(\w_E^n W)=0$ according to Theorem \ref{theorem modular kirillov}, contradicting that $\w_E^n W$ reaches an element of $O_E^\times$. Hence $W\in \Wh(V_0,\psi_E)_e$ and the statement follows.
\end{proof}

\vspace{-0,5cm}
\bibliographystyle{plain}

\def\cprime{$'$}

Data sharing not applicable to this article as no datasets were generated or analysed.

\end{document}